\documentclass[a4,11pt]{amsart}%
\usepackage{amsmath,amssymb,amsfonts,enumerate,amsthm, amscd,}
\usepackage[all]{xy}
\usepackage{graphicx}
\usepackage{amsmath}
\usepackage{amsfonts}
\usepackage{amssymb}
\usepackage{hyperref}
\providecommand{\U}[1]{\protect\rule{.1in}{.1in}}

\newtheorem{thm}{Theorem}[section]
\newtheorem{cor}[thm]{Corollary}

\newtheorem{prop}[thm]{Proposition}
\newtheorem{defn}[thm]{Definition}

\newtheorem{exam}[thm]{Example}
\newtheorem{rem}[thm]{Remark}

\newcommand{\field}[1]{\mathbb{#1}}

\newcommand{\Z }{\field{Z}}
\theoremstyle{definition}

\theoremstyle{remark}

\theoremstyle{Definition}

\begin{document}

\title{On $1$-absorbing $\delta$-primary ideals}

\author{Abdelhaq El Khalfi}
\address{Abdelhaq El Khalfi  \\Laboratory of Modelling and Mathematical Structures \\Department of Mathematics, Faculty of Science and Technology of Fez, Box 2202,
University S.M. Ben Abdellah Fez, Morocco.
$$ E-mail\ address:\  abdelhaq.elkhalfi@usmba.ac.ma $$}

\author{Najib Mahdou}
\address{Najib Mahdou\\Laboratory of Modelling and Mathematical Structures \\Department of Mathematics, Faculty of Science and Technology of Fez, Box 2202,
University S.M. Ben Abdellah Fez, Morocco.
$$E-mail\ address:\ mahdou@hotmail.com$$}

\author{\"{U}NSAL TEK\.{I}R}
\address{\"{U}NSAL TEK\.{I}R\\
Department of Mathematics, Marmara University, Istanbul, Turkey.
\[
E-mail\ address:\ utekir@marmara.edu.tr
\]
}

\author{SUAT KO\c{C}}
\address{SUAT KO\c{C}\\
	Department of Mathematics, Marmara University, Istanbul, Turkey.
	\[
	E-mail\ address:\ suat.koc@marmara.edu.tr
	\]
}



\begin{abstract}
Let $R$ be a commutative ring with nonzero identity. Let $\mathcal{I}(R)$ be the set of all ideals of $R$ and let $\delta : \mathcal{I}(R)\longrightarrow \mathcal{I}(R)$ be a function. Then $\delta$ is called an expansion function of ideals of $R$ if whenever $L, I, J$ are ideals of R with $J \subseteq I$, we have $L \subseteq \delta( L)$ and $\delta(J)\subseteq \delta(I)$. Let $\delta$ be an expansion
function of ideals of $R$. In this paper, we introduce and investigate a new class of ideals  that is closely related to the class of $\delta$-primary ideals. A proper ideal $I$ of $R$ is said to be a $1$-absorbing $\delta$-primary ideal if whenever nonunit elements $a,b,c \in R $ and $abc\in I$, then $ab \in I$ or $c\in \delta(I).$ Moreover, we give some basic properties of this class of ideals and we study the $1$-absorbing $\delta$-primary ideals of the localization of rings, the direct product of rings and the trivial ring extensions.
\end{abstract}

\maketitle


\bigskip



\section{Introduction}
Throughout this paper, all rings are assumed to be commutative with  nonzero identity and all modules are nonzero unital. If $R$ is a ring, then $\sqrt{I}$ denotes the radical of an ideal $I$ of $R$, in the sense of \cite[page 17]{Kap}. Let also Spec$(R)$ denotes the set of all prime ideals of $R$.

The prime ideal, which is an important subject of ideal theory, has been widely studied by various authors. Among the many recent generalizations of the notion of prime ideals in the literature, we find the following, due to Badawi \cite{B}. A proper ideal $I$ of $R$ is said to be a $2$-absorbing ideal if whenever $a,b,c\in R$ and $abc\in I$, then $ab\in I$ or $ac\in I$ or $bc\in I$. In this case $\sqrt{I}=P$ is a prime ideal with $P^2\subseteq I$ or $\sqrt{I}=P_1\cap P_2$ where $P_1,P_2$ are incomparable prime ideals with $\sqrt{I}^2 \subseteq I$, cf. \cite[Theorem 2.4]{B}. Recently, Badawi and Yetkin \cite{Badawi} consider a new class of ideals called the class of $1$-absorbing primary ideals. A proper ideal $I$ of a ring $R$ is called a $1$-absorbing primary ideal of $R$ if whenever nonunit elements $a, b, c \in R$ and $abc \in I$, then $ab \in I$ or $c\in \sqrt{I}$. In \cite{YNN}, A. Yassine et. al introduced the concept of $1$-absorbing prime ideals which is a generalization of prime ideals. A proper ideal $I$ of $R$ is a $1$-absorbing prime ideal if whenever we take nonunit elements $a,b,c\in R$ with $abc\in I$, then $ab\in I$ or $c\in I$. In this case $\sqrt{I}=P$ is a prime ideal, cf. \cite[Theorem 2.3]{YNN}. And if $R$ is a ring in which exists a $1$-absorbing prime ideal that is not prime, then $R$ is a local ring, that is a ring with one maximal ideal.

Let $\mathcal{I}(R)$ be the set of all ideals of a ring $R$. Zhao \cite{Zh} introduced the concept of expansion of ideals of $R$. We recall from \cite{Zh} that a function $\delta : \mathcal{I}(R)\longrightarrow \mathcal{I}(R)$ is called an expansion function of ideals of $R$ if whenever $L, I, J$ are ideals of $R$ with $J \subseteq I$, we have $L \subseteq \delta( L)$ and $\delta(J)\subseteq \delta(I)$. Note that there are explanatory examples of expansion functions included in \cite[Example 1.2]{Zh} and \cite[Example 1]{BF}.
In addition, recall from \cite{Zh} that a proper ideal $I$ of $R$ is said to be a $\delta$-primary
ideal of $R$ if whenever $a, b \in R$ with $ab \in I$, we have $a \in I$ or $b \in \delta(I)$, where $\delta$ is an
expansion function of ideals of $R$. Also, recall from \cite{BSY} that a proper ideal $I$ of $R$ is called a $\delta$-semiprimary  ideal of $R$ if $ab \in I$  implies $a \in \delta(I)$ or $b \in \delta(I)$. In this paper, we introduce and investigate a new concept of ideals  that is closely related to the class of $\delta$-primary ideals. A proper ideal $I$ of $R$ is said to be a $1$-absorbing $\delta$-primary ideal if whenever nonunit elements $a,b,c \in R $ and $abc\in I$, then $ab \in I$ or $c\in \delta(I).$ For example, let
$\delta : \mathcal{I}(R)\longrightarrow \mathcal{I}(R)$ such that $\delta(I)=\sqrt{I}$ for each ideal $I$ of $R$. Then $\delta$ is an expansion function of ideals of $R$, and hence a proper ideal $I$ of $R$ is a $1$-absorbing $\delta$-primary  ideal of $R$ if and only if $I$ is a $1$-absorbing primary ideal of $R$. Among many results in this paper  are given to disclose the relations between this new class and others that already exist. The reader may find it helpful to keep in mind the implications noted in the following figure.
\vspace{-2cm}
\setlength{\unitlength}{1mm}
\begin{center}
\setlength{\unitlength}{1.7mm} \begin{picture}(140,27)(17,7)
\put(34.5,18){$1$-absorbing prime ideal}
\put(64,18){$1$-absorbing $\delta$-primary ideal}
\put(17,18){prime ideal}
\put(29,18.4){\vector(1,0){4.5}}
\put(59,18.4){\vector(1,0){4.5}}
\end{picture}
\end{center}
\vspace{-1.5cm}

Among other things, we give an example of $1$-absorbing $\delta$-primary ideal that is not $1$-absorbing prime ideal (Example \ref{exam1}). Also, we show $($Theorem \ref{thm2}$)$ that if a ring $R$ admits a $1$-absorbing $\delta$-primary ideal of $R$ that is not a $\delta$-primary ideal, then $R$ is a local ring. Moreover,  we prove that if $R$ is a chained ring with maximal ideal $M$, then the only $1$-absorbing $\delta$-primary ideals of $R$ are $M^2$ and the $\delta$-primary ideals of $R$ (Theorem \ref{thm4}). Finally, we give an idea about some $1$-absorbing $\delta$-primary ideal of the localization of rings, the direct product of rings and the trivial ring extensions.

\section{Main Results}

We start this section by the following definition.
\begin{defn}
A proper ideal $I$ of a ring $R$ is called a $1$-absorbing $\delta$-primary ideal if whenever $abc\in I$ for some nonunit elements $a,b,c\in I$, then $ab\in I$ or $c\in \delta(I)$.
\end{defn}

\begin{rem}\label{rem1}
Let $R$ be a ring, $I$ a proper ideal of $R$ and $\delta$ be an expansion function of $\mathcal{I(R)}$.
\begin{enumerate}
\item[$(1)$] If $\delta(I)=I$, then $I$ is a $1$-absorbing $\delta$-primary ideal of $R$ if and only if it is a $1$-absorbing prime ideal.
\item[$(2)$] If $\delta(I)=\sqrt{I}$, then $I$ is a $1$-absorbing $\delta$-primary ideal of $R$ if and only if it is a $1$-absorbing primary ideal.
\item[$(3)$] Every $1$-absorbing prime ideal is a $1$-absorbing $\delta$-primary ideal.
\item[$(4)$] Every $\delta$-primary ideal is a $1$-absorbing $\delta$-primary ideal.
\item[$(5)$] Let $\gamma$ be an expansion function of the ideal of $R$ such that $\delta(I) \subseteq \gamma(I)$. If $I$ is a $1$-absorbing $\delta$-primary ideal of $R$, then $I$ is a $1$-absorbing $\gamma$-primary ideal of $R$.
\end{enumerate}
\end{rem}

Next, we give an example of a $1$-absorbing $\delta$-primary ideal that is not a $1$-absorbing prime ideal.
\begin{exam}\label{exam1}
Let $R:=K[[X_1,X_2,X_3]]$ be a ring of formal power series where $K$ is a field. Consider the expansion function $\delta: \mathcal{I(R)}\longrightarrow \mathcal{I(R)}$ defined by $\delta(I)=I+M$ where $M=(X_1,X_2,X_3)$ is the maximal ideal of $R$. Let $I=(X_1X_2X_3)$ be an ideal of $R$. Thus, $I$ is not a $1$-absorbing prime ideal of $R$ since $X_1X_2X_3\in I$ but neither $X_1X_2\in I$ nor $X_3\in I$. Now, let $x,y,z$ be nonunit elements of $R$ such that $xyz\in I$. Clearly $I$ is a $1$-absorbing $\delta$-primary because $z\in \delta(I)=M$.
\end{exam}

\begin{prop}
	(i) Every 1-absorbing $\delta$-primary ideals is also a 2-absorbing $\delta$-primary ideal of $R$.
	
	(ii) Let $I$ be a 1-absorbing $\delta$-primary ideal of $R$ and $\delta(I)$ be a radical ideal, that is, $\sqrt{\delta(I)}=\delta(I)$. Then $I$ is a $\delta$-semiprimary ideal of $R$.
\end{prop}

\begin{proof}
	(i) Let $abc \in I$ for some $a,b,c \in R$. If at least one of $a,b,c$ is a unit of $R$, then we are done. So assume that $a,b,c$ are nonunits $R$. Since $I$ is a 1-absorbing $\delta$-primary ideal of $R$, we get $ab \in I$ or $c \in \delta(I)$, which implies that $ab \in I$ or $ac \in \delta(I)$ or $bc \in \delta(I)$. Therefore, $I$ is a 2-absorbing $\delta$-primary ideal of $R$. 
	
	(ii) Suppose that $ab \in I$ for some $a,b \in R$. Then we may assume that $a,b$ are nonunits. Thus $a^2b \in I$ implies that $a^2 \in I$ or $b \in \delta(I)$. Then we have $a \in \sqrt{I} \subseteq \sqrt{\delta(I)}=\delta(I)$ or $b \in \delta(I)$. Hence, $I$ is a $\delta$-semiprimary ideal of $R$.
\end{proof}

The converse of previous theorem (i) is not true in general. See the following example.

\begin{exam}\textbf{(2-absorbing $\delta$-primary ideal that is not 1-absorbing $\delta$-primary ideal)}
	
	Let $R=\Z$, $I=pq\Z$, where $p \neq q$ are prime numbers, and $\delta(I)=I+p\Z$. Since $I$ is a 2-absorbing ideal, so is 2-absorbing $\delta$-primary. However, it is easy to see that $ppq \in I$, $p^2 \notin I$ and $q \notin \delta(I)$. Thus, $I$ is not a 1-absorbing $\delta$-primary ideal of $R$.

\end{exam}

In the next result, we show that if a ring $R$ admits a 1-absorbing $\delta$-primary ideal that is not a $\delta$-primary ideal, then $R$ is a local ring.
\begin{thm}\label{thm2}
Let $\delta$ be an ideal expansion. Suppose that a ring $R$ admits a $1$-absorbing $\delta$-primary ideal that is not a $\delta$-primary ideal. Then $R$ is a local ring.
\end{thm}
\begin{proof}
 Assume that $I$ is a $1$-absorbing $\delta$-primary ideal that is not a $\delta$-primary ideal of $R.$  Hence there exist nonunit elements $a, b \in R$ such that $ab\in I$,  $a \notin I$ and $b \notin \delta(I)$. Let $d$ be a nonunit element of $R .$ As $dab \in I$, $I$ is a $1$-absorbing $\delta$-primary ideal of $R$ and $b\notin \delta(I),$ we conclude that $da \in I .$ Let $c$ be a unit element of $R .$ Suppose that $d+c$ is a nonunit element of $R .$ Since $(d+c) ab \in I$, $I$ is a $1$-absorbing $\delta$-primary ideal of $R$ and $b \notin \delta(I)$, we get that $(d+c)a=da+c a \in I .$ Since $da \in I,$ we conclude that $a \in I,$ which gives a contradiction. Hence, $d+c$ is a unit element of $R$. Now, the result follows from \cite[Lemma 1]{Badawi}.
\end{proof}

Next, we give a method to construct $1$-absorbing $\delta$-primary ideals that are not $\delta$-primary ideals.
\begin{thm}\label{thm1}
Let $R$ be a local ring with maximal ideal $M$ and $\delta$ be an ideal expansion. Let $x$ be a nonzero prime element of $R$ such that $\delta (xM) \subsetneq M$. If $x\in \delta(xM)$, then $xM$ is a $1$-absorbing $\delta$-primary ideal of $R$ that is not a $\delta$-primary ideal of $R$.
\end{thm}
\begin{proof}
First, we will show that $xM$ is a $1$-absorbing $\delta$-primary ideal of $R$. Assume that $abc \in xM$ for some nonunit elements $a,b,c \in R$. If $ab\notin xM$, then $a\notin xR$ and $b\notin xR$, so $ab\notin xR$ because $x$ is a prime element of $R$. Moreover, the fact that $abc\in xR$ and $ab\notin xR$ implies that $c\in xR\subseteq \delta(xM)$. Now, we prove that $xM$ is not a $\delta$-primary ideal of $R$. By hypothesis, we can pick an element $a\in M\setminus \delta(xM)$, hence $xa\in xM$. However, $x\notin xM$ since $x$ is an irreducible element of $R$ by \cite[Lemma 2]{Badawi}. Which implies that $xM$ is not a $\delta$-primary ideal, this completes the proof.
\end{proof}

\begin{thm}\label{thm3}
Let $I$ be a $1$-absorbing $\delta$-primary ideal of a ring $R$ where $\delta$ is an ideal expansion and let $d \in R\setminus I$ be a nonunit element of $R$. Then $(I:d)=\{x\in R\mid  dx\in I\}$ is a $\delta$-primary ideal of $R$.
\end{thm}
\begin{proof}
 Suppose that $ab\in (I:d)$ for some elements $a,b\in R$. Without loss of generality, we may assume that $a$ and $b$ are nonunit elements of $R$. Suppose that $a\notin (I:d)$. Since $dab\in I$ and $I$ is a $1$-absorbing $\delta$-primary ideal of $R$, we conclude that $b\in \delta(I)$. So, $b\in \delta((I:d))$ and this completes the proof.
\end{proof}

\begin{prop}\label{pro1} Let $R$ be a ring, $\delta$ an ideal expansion and $I$ be a proper ideal of $R$. If $I$ is a $1$-absorbing $\delta$-primary ideal of $R$, then either $I$ is a $\delta$-semiprimary ideal of $R$ or $R$ is local, say with maximal ideal $M$, such that $M^2\subseteq I$.
\end{prop}

\begin{proof} If $R$ is not local, then Theorem \ref{thm2} implies that $I$ is $\delta$-primary and so $I$ is a $\delta$-semiprimary ideal of $R$. Now, assume that $R$ is local with maximal ideal $M$ such that $I$ is not a $\delta$-semiprimary ideal of $R$. Since $I$ is proper, we infer that $I\subseteq M$. Moreover, there are $a,b\in M\setminus \delta(I)$ such that $ab\in I$. To prove that $M^2\subseteq I$, it suffices to show that $xy\in I$ for all $x,y\in M$. Let $x,y\in M$. Then $xyab\in I$. Since $xy,a,b\in M$, $b\not\in \delta(I)$ and $I$ is a $1$-absorbing $\delta$-primary ideal, we conclude that $xya\in I$. Again, since $x,y,a\in M$, $a\not\in \delta(I)$ and $I$ is a $1$-absorbing $\delta$-primary ideal, we have that $xy\in I$.
\end{proof}

Recall that a ring $R$ is a chained ring if the set of all ideals of $R$ is linearly ordered by inclusion. Moreover, $R$ is said to be an arithmetical ring if $R_M$ is a chained ring for each maximal ideal $M$ of $R$. We next determinate the $1$-absorbing $\delta$-primary ideals of a chained ring.
\begin{thm}\label{thm4}
Let $R$ be a chained ring with maximal ideal $M$, $\delta$ an ideal expansion and $I$ be a proper ideal of $R$ such that $I\neq M^2$. Then $I$ is a $1$-absorbing $\delta$-primary ideal of $R$ if and only if $I$ is a $\delta$-primary ideal of $R$.
\end{thm}
\begin{proof}
We need only prove the ``only if'' assertion. Let $I$ be a $1$-absorbing $\delta$-primary ideal. Thus, Proposition \ref{pro1} gives that either $I$ is a $\delta$-semiprimary ideal of $R$ or $M^2\subseteq I$. First, assume that $I$ is a $\delta$-semiprimary ideal of $R$ and $ab\in I$ for some nonunit elements $a,b\in R$ such that $b\notin \delta(I)$. Hence, $a\in \delta(I)$. Now, since $R$ is a chained ring, we conclude that $a\in bR$ and thus $a=br$ for some nonunit element $r\in R$. As $brb\in I$, $b\notin \delta(I)$ and $I$ is a $1$-absorbing $\delta$-primary ideal of $R$, we conclude that $a=br\in I$. Which gives that $I$ is a $\delta$-primary ideal of $R$. Now, we suppose that $M^2\subseteq I$. We may assume that $M\neq I$. Thus, we can pick $a\in M\setminus I$ and $b\in I\setminus M^2.$ Then $b\in aR$ since $R$ is a chained ring. So, $b=ar$ for some nonunit element $r\in R$ and thus $b\in M^2$, a contradiction.  This completes the proof.
\end{proof}

In view of Theorem \ref{thm4}, we have the following result.
\begin{cor}
Let R be an arithmetical ring with Jacobson radical $M$ and $I$ be a proper ideal of $R$ such that $I\neq M^2$. Then $I$ is a $1$-absorbing $\delta$-primary ideal of $R$ if and only if $I$ is a $\delta$-primary ideal of $R$.
\end{cor}
\begin{proof}
Assume that $R$ is local with maximal ideal $M$. Since $R$ is an arithmetical ring, we conclude that $R = R_M$ is a chained ring and thus the claim follows from Theorem \ref{thm4}. In the remaining case, suppose that $R$ is not a local ring. Then the result follows by Theorem \ref{thm2}.
\end{proof}
\begin{prop}
Let $R$ be a local ring with principal maximal ideal $M$, $\delta$ an ideal expansion and $I$ be a proper ideal of $R$. Then $I$ is a $1$-absorbing $\delta$-primary ideal of $R$ if and only if either $I$ is a $\delta$-primary ideal of $R$ or $M^2\subseteq I$.\\ In addition, if $\sqrt{I}\subseteq \delta(I)$ then $I$ is a $1$-absorbing $\delta$-primary ideal of $R$ if and only if $I$ is a $\delta$-primary ideal of $R$.
\end{prop}
\begin{proof}
By Remark \ref{rem1}(4) and Proposition \ref{pro1}, we need only prove that if $I$ is a $1$-absorbing $\delta$-primary ideal of $R$ which is a $\delta$-semiprimary ideal then $I$ is a $\delta$-primary ideal (along with the hypothesis that $R$ be a local ring with principal maximal ideal $M$). Also, we may assume that $\delta(I)\neq R$. Set $M=xR$ and let $a$ and $b$ be nonunit elements of $R$ such that $b\notin \delta(I)$ and $ab \in I$. Since $I$ is a $\delta$-semiprimary ideal of $R$, we get that $a \in \delta(I)$. Moreover, $a=rx$ for some $r\in R$. If $r$ is a unit element of $R$ then $M=\delta(I)$ and thus $I$ is a $\delta$-primary ideal. If $r$ is a nonunit element of $R$ then $rxb=ab\in I$. That implies $a=rx\in I$ since $I$ is a $1$-absorbing $\delta$-primary ideal. This completes the proof. The in addition statement is clear.
\end{proof}

\begin{prop}
Let $\{J_i \mid i \in D\}$ be a directed set of $1$-absorbing $\delta$-primary ideals of $R$, where $\delta$ is an ideal expansion. Then the ideal $J = \cup _{i\in D}J_i$ is a $1$-absorbing $\delta$-primary ideal of $R$.
\end{prop}
\begin{proof}
Let $abc \in J$, then $abc\in J_{i}$ for some $i\in D$. Since $J_{i}$ is a $1$-absorbing $\delta$-primary ideal of $R$, $ab\in J_{i}$ or $c\in \delta(J_{i})\subseteq \delta(J)$. Hence, $J$ is a $1$-absorbing $\delta$-primary ideal of $R$.
\end{proof}
\begin{prop}
Let $I$ be a $1$-absorbing $\delta$-primary ideal of $R$ such that $\sqrt{\delta(I)}=\delta(\sqrt{I})$, where $\delta$ is an ideal expansion. Then,
$\sqrt{I}$ is a $\delta$-primary ideal of $R$.
\end{prop}
\begin{proof}
Let $ab\in \sqrt{I}$ such that $a\notin \sqrt{I}$. Hence, there exists a positive integer $n$ such that $(ab)^n\in I$. So, $a^ma^mb^n \in I$ for some positive integer $m$. Since $I$ is a $1$-absorbing $\delta$-primary ideal of $R$ and $a^{2m}\notin I$, we conclude that $b^n\in \delta(I)$. That implies $b\in \sqrt{\delta(I)}=\delta(\sqrt{I})$ and so $\sqrt{I}$ is a  $\delta$-primary ideal of $R$.
\end{proof}
\begin{prop}
Let $I$ be a proper ideal of a ring $R$ and $\delta$ be an ideal expansion such that $\delta(\delta(I))=\delta(I)$. Then $\delta(I)$ is a $1$-absorbing $\delta$-primary ideal of $R$ if and only if $\delta(I)$ is a $1$-absorbing prime ideal of $R$
\end{prop}
\begin{proof}
By Remark \ref{rem1}(3), we need only prove the ``only if" assertion. Let $abc\in \delta(I)$ for some nonunit elements $a,b,c\in R$. Hence $ab\in\delta(I)$ or $c\in \delta(\delta(I))=\delta(I)$. Thus $\delta(I)$ is a $1$-absorbing prime ideal of $R$.
\end{proof}
\begin{prop}\label{prop1-abs}
Let $R$ be a ring, $I$ a proper ideal of $R$ and $\delta$ be an ideal expansion. Then $I$ is a $1$-absorbing $\delta$-primary ideal if and only if whenever $I_1I_2I_3\subseteq I$ for some proper ideals $I_1$, $I_2$ and $I_3$ of $R$, then $I_1I_2\subseteq I$ or $I_3\subseteq \delta(I)$.
\end{prop}
\begin{proof}
It suffices to prove the ``if'' assertion. Suppose that $I$ is a $1$-absorbing $\delta$-primary ideal and let $I_1$, $I_2$ and $I_3$ be proper ideals of $R$ such that $I_1I_2I_3\subseteq I$ and $I_3\not\subseteq \delta(I)$. Thus $abc\in I$ for every $a\in I_1$, $b\in I_2$ and $c\in I_3\setminus \delta(I)$. Since $I$ is a $1$-absorbing $\delta$-primary ideal, we then have $I_1I_2\subseteq I$, as desired.
\end{proof}
Recall from \cite{Zh} that an ideal expansion $\delta$ is said to be intersection preserving if $\delta(I_1\cap I_2\cap ...\cap I_n)=\delta(I_1)\cap \delta(I_2)\cap ...\cap \delta(I_n)$ for any ideals $I_1,...,I_n$ of $R$.
\begin{prop}\label{proi}
Let $\delta$ be an intersection preserving ideal expansion. If $I_1, I_2,..., I_n$ are $1$-absorbing $\delta$-primary ideals of $R$, and $\delta(I_i)=P$ for all $i\in \{1,2,...,n\}$, then $I_1\cap I_2\cap...\cap I_n$ is a $1$-absorbing $\delta$-primary ideal of $R$.
\end{prop}
\begin{proof}
Let $abc\in J=I_1\cap I_2\cap...\cap I_n$ such that $ab\notin J$. Let $i\in \{1,2,...,n\}$ such that $ab\notin I_i$. Since $abc\in I_i$ and $I_i$ is a $1$-absorbing $\delta$-primary ideal, we conclude that $c\in \delta(I_i)=\delta(J)$. Therefore, $J$ is a $1$-absorbing $\delta$-primary ideal of $R$.
\end{proof}

\begin{prop}\label{princ}
Let $R$ be a ring and $\delta$ be an expansion function of $\mathcal{I}(R)$. Then the following statements
are equivalent:
\begin{enumerate}
\item[$(1)$] Every proper principal ideal is a $1$-absorbing $\delta$-primary ideal of $R$.
\item[$(2)$] Every proper ideal is a $1$-absorbing $\delta$-primary ideal of $R$.

\end{enumerate}
\end{prop}
\begin{proof}
Assume that $(1)$ holds and let $I$ be a proper ideal of $R$. Let $a,b,c$ be nonunit elements of $R$ such that $abc\in I$. Hence $abc\in abcR$ which implies that $ab\in abcR\subseteq I$ or $c\in abcR\subseteq \delta(I)$. Therefore $I$ is a $1$-absorbing $\delta$-primary ideal of $R$. The converse is clear.
\end{proof}

An expansion function $\delta$ of $\mathcal{I}(R)$ is said to satisfy \textit{condition $(*)$} if $\delta(I)\neq R$ for each proper ideal $I$ of $R$. Note that the identity function and radical operation are examples of expansion functions satisfying condition $(*)$.

\begin{thm}\label{char}
	Let $R$ be a ring and $\delta$ an expansion function of $\mathcal{I}(R)$ satisfying condition $(*)$ and $\delta(Jac(R))=Jac(R)$. Suppose that $\delta(xI)=x\delta(I)$ for every proper ideal $I$ of $R$ and every $x \in R$. The following statements are equivalent.

	(i) Every proper principal ideal is a 1-absorbing $\delta$-primary ideal of $R$.
	
	(ii) Every proper ideal is a 1-absorbing $\delta$-primary ideal of $R$.
	
	(iii) $R$ is local with $Jac(R)^2=(0)$, where $Jac(R)=m$ is the unique maximal ideal.
\end{thm}

\begin{proof}
	$(i)\Leftrightarrow(ii)$ Follows from Proposition \ref{princ}.

	$(i)\Rightarrow(iii)$ Assume that every proper ideal is a 1-absorbing $\delta$-primary ideal $R$. Choose $x,y \in Jac(R)$. Now, we will show that $xy=0$. If $x$ or $y$ is zero, then we are done. Assume that $x,y\neq 0$. Since $x^2y \in (x^2y)$ and $(x^2y)$ is a 1-absorbing $\delta$-primary ideal, we conclude that $x^2 \in (x^2y)$ or $y \in \delta((x^2y))=y\delta((x^2))$. Suppose that $y \in y\delta((x^2))$. Then there exists $a \in \delta((x^2))\subseteq\delta(Jac(R))=Jac(R)$ such that $y=ya$. Which implies that $y(1-a)=0$. Since $1-a$ is unit, we have $y=0$, which is a contradiction. Thus we have, $x^2 \in (x^2y)$. Then we can write $x^2=rx^2y$ for some $r \in R$. This implies that $x^2(1-ry)=0$. Since $1-ry$ is unit, we have $x^2=0$. Likewise, we get $y^2=0$. Now, choose another $z \in Jac(R)$. Since $xyz \in (xyz)$ and $(xyz)$ is a 1-absorbing $\delta$-primary, we get $xy \in (xyz)$ or $z \in\delta((xyz))=z\delta((xy))$. First, assume that $xy \in (xyz)$. Then there exists $r \in R$ such that $xy=rxyz$, which implies that $xy(1-rz)=0$. Since $1-rz$ is unit, we have $xy=0$ which completes the proof. Now, assume that $xy\notin (xyz)$, that is, $z \in \delta((xyz))=z\delta((xy))$. Then there exists $a \in \delta((xy))\subseteq Jac(R)$ such that $z=za$. This implies that $z(1-a)=0$ so that $z=0$. Now, choose $z=x+y$. Then by above argument, we have either $xy=0$ or $z=x+y=0$. If $z=x+y=0$, then we have $x=-y$ and so $xy=-y^2=0$ which completes the proof. Therefore, $Jac(R)^2=(0)$. 
	
	Now, we will show that $R$ is a local ring. Choose maximal ideals $M_{1},M_{2}$ of $R$. Now, put $I=M_{1}\cap M_{2}$. Since $M_{1}^2M_{2}\subseteq I$ and $I$ is a 1-absorbing $\delta$-primary ideal, we have either $M_{1}^2\subseteq I\subseteq M_{2}$ or $M_{2}\subseteq\delta(I)\subseteq \delta(M_{1})$. \textbf{Case 1:} Suppose that $M_{1}^2\subseteq M_{2}$. Since $M_{2}$ is prime, clearly we have $M_{1}\subseteq M_{2}$ which implies that $M_{1}=M_{2}$. \textbf{Case 2:} Suppose that $M_{2}\subseteq \delta(M_{1})$. Since $\delta$ satisfies condition $(*)$, $\delta(M_{1})$ is proper. As $M_{1}\subseteq\delta(M_{1})$ and $M_{1}$ is a maximal ideal, we have $M_{1}=\delta(M_{1})$. Then we get $M_{2}\subseteq M_{1}$, which implies that $M_{1}=M_{2}$. Therefore, $R$ is a local ring.

	$(iii)\Rightarrow(i)$ Suppose that $R$ is a local ring with $Jac(R)^2=(0)$. Let $I$ be a proper ideal of $R$ and $abc \in I$ for some nonunits $a,b,c \in R$. Then $a,b,c \in Jac(R)$ since $R$ is local. As $Jac(R)^2=(0)$, we have $ab=0 \in I$. Therefore, $I$ is a 1-absorbing $\delta$-primary ideal of $R$.
\end{proof}

It can be easily seen that, in Theorem \ref{char}, $(iii)$ always implies $(i)$ without any assumption on $\delta$. But we give some examples showing that the converse is not true if we drop the aforementioned assumptions on $\delta$. 

\begin{exam}
	Let $R=\Z_{p^3}$, where $p$ is a prime number and $\delta(I)=R$ for every proper ideal $I$ of $R$. Note that $\delta$ does not satisfy condition $(*)$ and note that every ideal $I$ of $R$ is 1-absorbing $\delta$-primary. Thus $Jac(R)^2\neq(0)$, while $R$ is a local ring.
\end{exam}

\begin{exam}
	Let $k$ be a field and consider the formal power series ring $R=k[[X]]$. Then $R$ is a local ring with unique maximal ideal $m=(X)$. Define expansion function $\delta$ as $\delta(I)=\sqrt{I}$ for every ideal $I$ of $R$. Then it is easy to see that every ideal of $R$ is a 1-absorbing $\delta$-primary ideal. Also, it is clear that $\delta$ satisfies condition $(*)$ and $\delta(Jac(R))=Jac(R)$ but not satisfy the condition $\delta(xI)=x\delta(I)$.  Furthermore, $Jac(R)^2\neq(0)$. Thus Theorem \ref{char} fails without assumption $\delta(xI)=x\delta(I)$.
\end{exam}

\begin{cor}
	Let $R$ be a ring. The following statements are equivalent. 
	
	(i) Every proper ideal is a 1-absorbing prime ideal of $R$.
	
	(ii) Every proper principal ideal is a 1-absorbing prime ideal of $R$.
	
	(iii) $R$ is local with $Jac(R)^2=(0)$.
\end{cor}

\begin{proof}
	$(i)\Leftrightarrow (ii)$ Follows from Proposition \ref{princ}.
	
	$(ii)\Rightarrow(iii)$ Let $\delta$ be the identity expansion function, that is, $\delta(I)=I$ for every ideal $I$ of $R$. Note that $\delta$ satisfies all axioms in Theorem \ref{char}. Then $R$ is a local ring with $Jac(R)^2=(0)$.
	
	$(iii) \Rightarrow(i)$ It is similar to Theorem \ref{char} $(iii)\Rightarrow(i)$.
\end{proof}

An ideal expansion $\delta$ is called a prime expansion if for any $1$-absorbing $\delta$-primary ideal $I$ of $R$, $\delta(I)$ is a prime ideal of $R$.

\begin{prop}
Let $R$ be a local ring with maximal ideal $M$ and $\delta$ be a prime expansion function of $\mathcal{I}(R)$. Assume that one of the following conditions holds:
\begin{enumerate}
\item[$(1)$] Spec$(R)=\{\delta(0)\}$.
\item[$(2)$] Spec$(R)=\{\delta(0), M\}$ and $\delta(0)M=0$.
\end{enumerate}
Then every proper ideal of $R$ is $1$-absorbing $\delta$-primary.
\end{prop}
\begin{proof}
Let $I$ be a proper ideal of $R$ and assume that $(1)$ holds. Since $\delta$ is a prime expansion, we conclude that $\delta(I)=\delta(0)$ is the maximal ideal of $R$. Clearly $I$ is a $1$-absorbing $\delta$-primary ideal of $R$. Now, assume that Spec$(R)=\{\delta(0), M\}$ and $\delta(0)M=0$. If $\delta(I)=M$, we have then $I$ is $1$-absorbing $\delta$-primary. In the remaining case, $\delta(I)=\delta(0)$. Let $abc\in I$ for some nonunit elements $a,b,c\in R$ such that $c\notin \delta(0)$. As $I\subseteq \delta(0)$, we get that either $a\in \delta(0)$ or $b\in \delta(0)$. Thus $ab=0\in I$ which gives that $I$ is a $1$-absorbing $\delta$-primary ideal of $R$.
\end{proof}
Let $f:R\rightarrow S$ be a ring homomorphism and $\delta$, $\gamma$ expansion functions of $\mathcal{I}(R)$ and $\mathcal{I}(S)$ respectively. Recall from \cite{BF} that $f$  is called a $\delta\gamma$-homomorphism if  $\delta(f^{-1}(I)) = f^{-1}(\gamma(I))$ for each ideal $I$ of $S$. Also note that if $f$ is a $\delta\gamma$-epimorphism and $I$ is an ideal of $R$ containing $ker(f)$, then  $\gamma(f(I)) = f(\delta(I))$.
\begin{thm}
 Let $f: R \rightarrow S$ be a ring $\delta\gamma$-homomorphism where $\delta$, $\gamma$ are expansion functions of $\mathcal{I}(R)$ and $\mathcal{I}(S)$ respectively. Suppose that f($a$) is nonunit in $S$ for every nonunit element $a$ in $R$. Then the following statements hold.
 \begin{enumerate}
\item[$(1)$] If $J$ is a $1$-absorbing $\gamma$-primary ideal of $S,$ then $f^{-1}(J)$ is a $1$-absorbing $\delta$-primary ideal of $R$.
\item[$(2)$] If $f$ is an epimorphism and $I$ is a proper ideal of $R$  containing $\operatorname{ker}(f),$ then $I$ is a $1$-absorbing $\delta$-primary ideal of $R$ if and only if $f(I)$ is a $1$-absorbing $\gamma$-primary ideal of $S$.
\end{enumerate}
\end{thm}
\begin{proof}
$(1)$  Assume that $abc \in f^{-1}(J),$ for some nonunit elements $a, b, c \in R .$ Then $f(a)f(b)f(c)\in J$. Thus $f(a)f(b) \in J$ or $f(c) \in \gamma(J),$ which implies that $ab \in f^{-1}(J)$ or $ c \in f^{-1}(\gamma(J))=\delta(f^{-1}(J)) .$ Therefore, $f^{-1}(J)$ is a $1$-absorbing $\delta$-primary ideal of $R$.\\
$(2)$ Suppose that $f(I)$ is an $1$-absorbing $\gamma$-primary ideal of $S$. Since $I = f^{-1} (f (I))$, we conclude that $I$ is a $1$-absorbing $\delta$-primary ideal of $R$ by $(1)$. Conversely, let $x, y, z$ be nonunit elements of $S$ with $xyz \in f(I)$. Then there exist $a,b,c \in R$ such that $x=f(a)$, $y=f(b)$ and $z=f(c)$ with $f(a b c)=x y z \in f(I) .$ Since $\operatorname{ker}(f) \subseteq I$, we then have $abc \in I .$ Since
$I$ is a $1$-absorbing $\delta$-primary ideal of $R$ and $abc \in I$, we conclude that $a b \in I$ or $c \in \delta(I)$  which gives that $x y \in f(I)$ or $z \in f(\delta(I))=\gamma(f(I))$. Thus $f(I)$ is a $1$-absorbing $\delta$-primary ideal of $S$.

\end{proof}

Let $\delta$ be an expansion function of $\mathcal{I(R)}$ and $I$ an ideal of $R$. Then the function $\bar{\delta} : \frac{R}{I}\longrightarrow \frac{R}{I}$ defined by $\bar{\delta}(\frac{J}{I}) = \frac{\delta(J)}{I}$ for all ideals $I \subseteq J$, becomes an expansion function of $\frac{R}{I}$. Then, we have the following result.
\begin{cor}
Let $R$ be a ring, $\delta$ an expansion function of $\mathcal{I(R)}$ and $I\subseteq J$ be proper ideals of $R$. Assume that $a+I$ is a nonunit element of $\frac{R}{I}$ for every nonunit element $a\in R$. Then J is a $1$-absorbing  $\delta$-primary ideal of $R$ if and only if $\frac{J}{I}$ is a $1$-absorbing $\bar{\delta}$-primary ideal of $\frac{R}{I}$.
\end{cor}
\begin{prop}
Let $S$ be a multiplicatively closed subset of a ring $R$ and $\delta_S$ an expansion function of $\mathcal{I}(S^{-1}R)$ such that $\delta_S(S^{-1}I)=S^{-1}(\delta(I))$ for each ideal $I$ of $R$. If $I$ is a 1-absorbing $\delta$-primary ideal of $R$ such that $I \cap S = \varnothing$, then $S^{-1}I$ is a 1-absorbing $\delta_S$-primary ideal of $S^{-1}R$.
\end{prop}
\begin{proof}
Let $I$ be a 1-absorbing $\delta$-primary ideal of $R$ such that $I\cap S = \varnothing$ and $\frac{a}{s}\frac{b}{t}\frac{c}{r}\in S^{-1}I$ for some nonunit elements $a, b, c \in R$ and $s,t,r\in S$ such that $\frac{a}{s}\frac{b}{t}\notin S^{-1}I$. Then $xabc \in I$ for some $x\in S$. Since $I$ is a $1$-absorbing $\delta$-primary and $xab \notin I$, we conclude that
$c \in \delta(I)$. Thus $\frac{c}{r}\in S^{-1}(\delta(I))=\delta_S(S^{-1}I)$ which completes the proof.
\end{proof}

Let $S$ be a multiplicatively closed subset of a ring $R$ and $I$ an ideal of $R$. The next example shows that if $S^{-1}I$ is a 1-absorbing $\delta_S$-primary ideal of $S^{-1}R$, then $I$ need not to be a 1-absorbing $\delta$-primary ideal of $R$.
\begin{exam}
Let $p\neq q$ be two prime numbers. Set $I=pq\Z$ and $\delta$ be an ideal expansion such that $\delta(I)=I+q\Z$ for each ideal $I$ of $\Z$. Clearly, $I$ is not a 1-absorbing $\delta$-primary ideal of $\Z$ because $qqp\in I$ but neither $q^2\in I$ nor $p\in \delta(I)$. Now, let $S=\Z\setminus p\Z$ and note that $S^{-1}I=S^{-1}(p\Z)$. Let $\frac{a}{r_1}\frac{b}{r_2}\frac{c}{r_3}\in S^{-1}I$ for some nonunit elements $\frac{a}{r_1},\frac{b}{r_2},\frac{c}{r_3}\in S^{-1}\Z$. Note that $\frac{x}{r}\in S^{-1}\Z$ is nonunit if and only if $x\in p\Z$. Thus $a\in p\Z$ and $b\in p\Z$. Which gives that $\frac{a}{r_1}\frac{b}{r_2}\in S^{-1}I $ and hence $S^{-1}I$ is a $1$-absorbing $\delta_S$-primary ideal.
\end{exam}

Let $R_{1}$ and $R_{2}$ be two rings, let $\delta_{i}$ be an expansion function of $\mathcal{I}(\mathcal{R}_{i})$ for each $i \in\{1,2\}$ and $R=R_{1} \times R_{2} .$ For a proper ideal $I_{1} \times I_{2},$ the function $\delta_{\times}$ defined by $\delta_{\times}(I_{1} \times I_{2})=$
$\delta_{1}(I_{1}) \times \delta_{2}(I_{2})$ is an expansion function of $\mathcal{I}(\mathcal{R}) .$ The following result characterizes the $1$-absorbing $\delta$-primary ideals of the direct product of rings.

\begin{thm}\label{thmp} Let $R_{1}$ and $R_{2}$ be rings, $R=R_{1} \times R_{2} $ and let $\delta_{i}$ be an expansion function of $\mathcal{I}(\mathcal{R}_{i})$ for $i =1,2$. Then the following statements are equivalent:
\begin{enumerate}
\item[$(1)$] $I$ is a $1$-absorbing $\delta_\times$-primary ideal of $R$.
\item[$(2)$] $I$ is a $\delta_\times$-primary ideal of $R$.
\item[$(3)$] Either $I=I_{1} \times R_{2}$, where $I_{1}$ is a $\delta_{1}$-primary ideal of $R_{1}$ or $I=R_{1} \times I_{2}$, where $I_{2}$ is a $\delta_{2}$-primary ideal of $R_{2}$ or $I=I_{1} \times I_{2},$ where $I_{1}$ and $I_{2}$ are proper ideals of $R_{1}, R_{2},$ respectively with $\delta_{1}(I_{1})=R_{1}$ and $\delta_{2}(I_{2})=R_{2}$.

\end{enumerate}

\end{thm}
\begin{proof}
$(1)\Leftrightarrow (2)$. This follows from Theorem \ref{thm2}.\\
 $(2) \Leftrightarrow(3)$ Let $I$ be a $\delta_\times$-primary ideal of $R .$ Hence $I$ has the form $I=I_{1} \times I_{2}$ where $I_{1}$ and $I_{2}$ are ideals of $R_{1}$ and $R_{2}$ respectively. Without loss of generality, we may assume that $I=I_{1} \times R_{2}$ for some proper ideal $I_{1}$ of $R_{1} .$ We show that $I_{1}$ is a $\delta$-primary ideal of $R_{1}$. Deny. Then there are $a, b \in R_{1}$ such that $a b \in I_{1}, a \notin I_{1}$ and $b \notin \delta_{1}(I_{1}) .$ Hence $(a, 1)(b, 1) \in I_{1} \times R_{2} .$ Which implies that $(a,1) \in I_{1} \times R_{2}$ or $(b,1) \in \delta_{\times}(I_{1} \times R_{2})$ and so $a \in I_{1}$ or $b \in \delta_1(I_{1}),$ which gives a contradiction. Now suppose that both $I_{1}$ and $I_{2}$ are proper. As $(1,0)(0,1) \in I_{1} \times I_{2}$ and $(1,0),(0,1) \notin I_{1} \times I_{2},$ we have $(1,0),(0,1) \in \delta_{\times}(I_{1} \times I_{2})=\delta_{1}(I_{1}) \times \delta_{2}(I_{2}) .$ Therefore $\delta_{1}(I_{1})=R_{1}$ and $\delta_{2}(I_{2})=R_{2}$. The converse is clear.
\end{proof}

The following example proves that the condition ``$\delta(I_i)=P$ for all $i\in \{1,2,...,n\}$" is necessary in Proposition \ref{proi}.

\begin{exam}
Consider $R=\Z\times \Z$, $I_1=4\Z\times \Z$ and $I_2=\Z\times 9\Z$. Let $\delta$ be an expansion function of $\mathcal{I}(\mathcal{\Z})$ such that for every ideal $I$ of $\Z$ we have $\delta(I)=\sqrt{I}+J$ where $J=2\Z$. Thus $\delta_\times(I_1)=2\Z\times \Z$ and $\delta_\times (I_2)=\Z\times\Z$. Moreover, $I_1$ and $I_2$ are $1$-absorbing $\delta_\times$-primary ideal. But $I_1\cap I_2=4\Z\times 9\Z$ is not a $1$-absorbing $\delta_\times$-primary ideal by Theorem \ref{thmp}.
\end{exam}

Let $A$ be a ring and $E$ an $A$-module. Then $A\ltimes E$, the {\it trivial} ({\it ring}) {\it extension of} $A$ {\it by} $E$, is the ring whose additive structure is that of the external direct sum $A \oplus E$ and whose multiplication is defined by $(a, e) (b, f) := (ab, af + be)$ for all $a,b \in A$ and all $e, f \in E$. (This construction is also known by other terminology and other notation, such as the {\it idealization} $A(+)E$.) The basic properties of trivial ring extensions are summarized in the books \cite{Huck}, \cite{Glaz}. Trivial ring extensions have been studied or generalized extensively, often because of their usefulness in constructing new classes of examples of rings satisfying various properties (cf. \cite{Anderson2, DEM, DMZ, km} ). In addition, for an ideal $I$ of $A$ and a submodule $F$ of $E$, $I\ltimes F$ is an ideal of $A\ltimes E$ if and only if $IE\subseteq F$. Moreover, for an
expansion function $\delta$ of $A$, it is clear that $\delta_\ltimes$ defined as $\delta_\ltimes(I\ltimes F) = \delta(I)\ltimes E$ is an expansion function of $A\ltimes E$. Also as usual, if $c\in A$ then $(F:c)=\{e\in E\mid ce\in F\}$.

\begin{thm}
Let $A$ be a ring, $E$ an $A$-module and $\delta$ be an expansion function of $\mathcal{I}(A)$. Let $I$ be an ideal of $A$ and $F$ a submodule of $E$ such that $IE\subseteq F$. Then the following statement hold:
\begin{enumerate}
\item[$(1)$] If $I\ltimes F$ is a $1$-absorbing $\delta_\ltimes$-primary ideal of $A\ltimes E$, then $I$ is a $1$-absorbing $\delta$-primary ideal of $A$.
\item[$(2)$] Assume that $(F:c)=F$ for every $c\in A\setminus I$. Then  $I\ltimes F$ is a $1$-absorbing $\delta_\ltimes$-primary ideal of $A\ltimes E$ if and only if $I$ is a $1$-absorbing $\delta$-primary ideal of $A$.
\end{enumerate}
\end{thm}
\begin{proof}
$(1)$ Assume that $I\ltimes F$ is a $1$-absorbing $\delta_{\ltimes}$-primary ideal of $A\ltimes E$ and let $a,b,c$ be nonunit elements of $A$ such that $abc\in I$. Thus $(a,0)(b,0)(c,0)=(abc,0)\in I\ltimes F$ which implies that $(a,0)(b,0)\in I\ltimes F$ or $(c,0)\in \delta_\ltimes(I\ltimes F)=\delta(I)\ltimes E$. Therefore $ab\in I$ or $c\in \delta(I)$ and so $(1)$ holds.\\
$(2)$ By $(1)$, it suffices to prove the "if" assertion. Let $(a,s), (b,t),(c,r)$ be nonunit elements of $A\ltimes E$ such
that $(a,s)(b,t)(c,r)=(abc,bcs+act+abr)\in I\ltimes F$. Clearly, $abc\in I$ and so $ab\in I$ or $c\in \delta(I)$ since $I$
 is a $1$-absorbing $\delta$-primary ideal of $A$. If $c\in \delta(I)$, then $(c,r)\in \delta(I)\ltimes E=\delta_\ltimes(I \ltimes F)$. Hence, we may assume that $c\notin\delta(I)$. Then $ab\in I$. As $bcs+act+abr\in F$ and $abr\in F$, we get that $bcs+act\in F$. This implies $bs+at\in (F:c)=F$ and so $(a,s)(b,t)=(ab,at+bs)\in I\ltimes F$. Therefore $I\ltimes F$ is a $1$-absorbing $\delta_\ltimes$-primary ideal of $A\ltimes E$.

\end{proof}
\begin{cor}
Let $A$ be a ring, $E$ an $A$-module and $\delta$ be an expansion function of $\mathcal{I}(A)$. Let $I$ be a proper ideal of $A$. Then $I\ltimes E$ is a $1$-absorbing $\delta_\ltimes$-primary ideal of $A\ltimes E$ if and only if $I$ is a $1$-absorbing $\delta$-primary ideal of $A$.
\end{cor}

\end{document}